\input ./style/arxiv-vmsta.cfg
\documentclass[numbers,compress,v1.0.1]{vmsta}

\volume{4}
\issue{1}
\pubyear{2017}
\firstpage{15}
\lastpage{24}
\doi{10.15559/16-VMSTA71}

\newtheorem{prop}{Proposition}
\newtheorem{lemma}{Lemma}
\newtheorem{corollary}{Corollary}
\newtheorem{theorem}{Theorem}

\theoremstyle{definition}
\newtheorem{remark}{Remark}

\newtheorem{definition}{Definition}

\startlocaldefs

\allowdisplaybreaks
\endlocaldefs

\begin{document}
\begin{frontmatter}

\title{Generalized fractional Brownian motion}
\author{\inits{M.}\fnm{Mounir}\snm{Zili}}\email{Mounir.Zili@fsm.rnu.tn}
\address{University of Monastir, Faculty of Sciences of Monastir,\\
Department of Mathematics, Avenue de l'Environnement,\\ 5000 Monastir, Tunisia}

\markboth{M. Zili}{Generalized fractional Brownian motion}

\begin{abstract}
We introduce a new Gaussian process, a generalization of both
fractional and subfractional Brownian motions, which could serve as a
good model for a larger class of natural phenomena. We study its main
stochastic properties and some increments characteristics. As an
application, we deduce the properties of nonsemimartingality, H\"older
continuity, nondifferentiablity, and existence of a local time.
\end{abstract}

\begin{keywords}
\kwd{Generalized fractional and subfractional Brownian motion}
\kwd{stationarity}
\kwd{Markovity}
\kwd{semimartingality}
\end{keywords}

\received{13 November 2016}
%
\accepted{14 December 2016}
\publishedonline{16 January 2017}
\end{frontmatter}

\section{Introduction}

It is well known that the two-sided fractional Brownian motion (tsfBm)
with Hurst parameter
$ H \in(0,1)$ is a
centered Gaussian process $B^H = \{ B_t^H, t \in{\mathbb
R} \} $,
defined on a probability space $(\varOmega, F, {\Bbb P}) $,
with the covariance function
\begin{equation}
\label{eq:1} Cov \bigl(B_t^H, B_s^H
\bigr) = \frac{1}{2} \bigl( | s |^{2H} + | t |
^{2H} - | t -s |^{2H} \bigr) .
\end{equation}

When $H = \frac{1}{2} $, $B^H$ is a two-sided Brownian
motion (tsBm).
The self-similarity and stationarity of the increments are two main
properties because of which fBm enjoyed success as a modeling tool in
many areas, such as finance, hydrology, biology, and telecommunications.

In \cite{TB}, the authors suggested another kind of extension of the
Bm, called the
subfractional Brownian motion (sfBm), which preserves most of the properties
of the fBm, but not the stationarity of the increments. It is a
centered Gaussian process $\xi^H = \{ \xi_t^H, t \in[0,
\infty) \} $
with the covariance function
\begin{equation}
\label{eq:1-1} S(t,s) = t^{2H} + s^{2H} - \frac{1}{2}
\bigl( (t+s)^{2H} + | t-s | ^{2H} \bigr),
\end{equation}
with $H \in(0,1)$. The case $H = \frac{1}{2}$ corresponds to the Bm.

The sfBm is intermediate between Brownian motion and fractional
Brownian motion
in the sense that it has properties analogous to those of fBm, but the
increments on nonoverlapping
intervals are more weakly correlated, and their covariance decays
polynomially at a higher rate. So the sfBm does not generalize the fBm.

One extension of the sfBm was introduced in \cite{SGh} and called the
generalized subfractional Brownian motion (GsfBm). It is a centered
Gaussian process starting from zero with covariance function
\[
G(t,s) = \bigl(t^{2H} + s^{2H}\bigr)^K -
\frac{1}{2} \bigl( (t+s)^{2HK} + | t-s |^{2HK} \bigr),
\]
with $H \in(0,1)$ and $K \in[1,2)$. The case $K =1$ corresponds to
the sfBm.

Another type of a generalized form of the sfBm was introduced in \cite
{MZROSE} and \cite{Zi-ElN} as a linear combination of a finite number
of independent subfractional Brownian motions. It was called the mixed
subfractional Brownian motion (msfBm). The msfBm is a centered mixed
self-similar\footnote{The mixed self-similarity property was introduced
by Zili \cite{MZ}.} Gaussian process and does not have stationary
increments. Both GsfBm and msfBm do not generalize the fBm.

On the other hand, in the literature, we find many different kinds of
extensions of the fBm, such as the multifractional Brownian motion \cite
{Pel}, the mixed fractional Brownian motion \cite{MZ}, and the
bifractional Brownian motion \cite{Houd}. But none of them generalizes
the sfBm.

In this paper, we introduce a new stochastic process, which is an
extension of both subfractional Brownian motion and fractional Brownian
motion. This process is completely different from all the other
extensions existing in the literature; we call it the generalized
fractional Brownian motion. More precisely, let us take two real
constants $a$ and $b$ such that
$(a, b) \neq(0,0)$ and $H \in(0,1)$.
\begin{definition}
A generalized fractional Brownian motion (gfBm) with parameters $a,b$,
and $H$,
is a process
$Z^H = \{ Z_t^H(a,b) ; t \ge0 \} =
\{ Z_t^H; t \ge0 \} $ defined on the probability
space $(\varOmega, F, {\Bbb P}) $ by
\begin{equation}
\label{eq:3} \forall t \in{\Bbb R}_+, \hspace{5mm} Z_t^H
= Z_t^H(a,b) = a B_t^H + b
B_{-t}^H ,
\end{equation}
where
$(B_t^H)_{t \in{\Bbb R}}$ is a two-sided fractional
Brownian motion with parameter $H$.
\end{definition}

If $a = 1, b= 0$, then $Z^H = \{ B^H_t; t \ge0 \}$, that
is, $Z^H$ is a fractional Brownian motion. If $a = b =
\frac{1}{\sqrt{2}}$, then it is easy to see, either by a direct
calculation using (\ref{eq:1}) or by Lemma~\ref{L1}, that the
covariance of process $Z^H$ is precisely $S(t,s)$ given by (\ref
{eq:1-1}). So, in this case, $Z^H$ is a subfractional Brownian motion. If $a = b = \frac
{1}{\sqrt{2}}$ and $H = \frac{1}{2}$ or if $ a= 1, b= 0
$, and
$H = \frac{1}{2}$, $G^H$ is clearly a~standard Brownian motion.

So the gfBm is, at the same time, a generalization of the fractional
Brownian motion, of the subfractional Brownian motion, and of course of
the standard Brownian motion. This is an important motivation for the
introduction of such a process since it allows to deal with a larger
class of modeled natural phenomena, including those with stationary or
nonstationary increments.

This paper contains two sections. In the first one, we investigate the
main stochastic properties of the gfBm. We show in particular that the
gfBm is a~Gaussian, self-similar, and non-Markov process (except in the
case where $H = 1/2$). The second section is devoted to the
investigation of some characteristics of increments of gfBms. As an
application, we deduce the properties of nonsemimartingality, H\"older
continuity, nondifferentiablity, and existence of a local time.

\section{The main properties}
By the Gaussianity of $B^H$and by Eq.~(\ref{eq:1}) we easily get the
following lemma.
\begin{lemma}
\label{L1}
The gfBm $(Z_t^H(a,b))_{t \in{\Bbb R}_+}$ satisfies the following
properties:
\begin{itemize}
\item $Z^H $ is a centered Gaussian process.
\item $\forall s \in{\Bbb R}_+,\ \forall t \in{\Bbb R}_+$,
\begin{align*}
&\operatorname{Cov} \bigl(Z_t^H(a,b), Z_s^H(a,b) \bigr)\\
&\quad = \frac{1}{2}(a+b)^2 \bigl( s ^{2H} + t^{2H} \bigr) - ab ( t+ s ) ^{2H} - \frac{a^2 + b^2}{2} | t- s|^{2H} .
\end{align*}
\item $\forall t \in{\Bbb R}_+$,
\begin{align*}
E \bigl( Z_t^H(a,b)^2 \bigr) = \bigl( a^2 + b^2 - \bigl(2^{2H}-2\bigr)ab \bigr) t ^{2H}.
\end{align*}
\end{itemize}
\end{lemma}

The following lemma deals with the self-similarity property.

\begin{lemma}
The gfBm is a self-similar process.
\end{lemma}

\begin{proof}
This follows from the fact that, for fixed $h >0$, the processes
$\{ Z_{ht}^H(a,b); t \ge0 \} $ and
$\{ h^{H} Z_t^H (a,b) ; t \ge0 \} $ are
Gaussian, centered, and have the same
covariance function.
\end{proof}

Let us now study the non-Markov property of the gfBm.

\begin{prop}
\label{p:1}
{\it For all $ H \in(0; 1 ) \setminus\{ \frac{1}{2} \} $
and $(a,b) \in{\mathbb R}^2 \setminus\{ (0,0) \}$,
$(Z_t^H(a,b))_{t \in{\Bbb R}_+}$ is not a Markov
process. }
\end{prop}

\begin{proof}
We will only prove the proposition in the case where $b \neq
0$; the result with $b = 0$ is known (see \cite{MZ} and the references therein).
The process $Z^H$ is a~centered Gaussian, and for all $t > 0$,
\[
\operatorname{Cov} \bigl( Z^H_t, Z^H_t \bigr)
= \bigl(a^2 + b^2 - \bigl(2^{2H}-2\bigr)ab\bigr)
t^{2H} > 0.
\]
Then, if $Z^H$ were a Markov process, according
to \cite{Re}, for all $s < t < u$, we would have
\begin{equation}
\operatorname{Cov} \bigl( Z^H_s,Z^H_u \bigr)
\operatorname{Cov} \bigl( Z^H_t, Z^H_t \bigr)
= \operatorname{Cov} \bigl( Z^H_s,Z^H_t \bigr)
\operatorname{Cov} \bigl( Z^H_t, Z^H_u \bigr)
.
\end{equation}

In the particular case where $1 < s = \sqrt{t} < t < u =
t^2$, we have
\begin{align*}
& \biggl[ \frac{1}{2} (a+b)^2\bigl(1+t^{-3H}\bigr)
-ab\bigl(1+t^{-3/2}\bigr)^{2H}-\frac{a^2+b^2}{2}
\bigl(1-t^{-3/2}\bigr)^{2H} \biggr]
\\
& \qquad \times \bigl[ a^2 + b^2) -
\bigl(2^{2H}-2\bigr) ab \bigr]
\\
& \quad = \biggl[ \frac{1}{2} (a+b)^2\bigl(1+t^{-H}
\bigr) -ab\bigl(1+t^{-1/2}\bigr)^{2H}-\frac{a^2+b^2}{2}
\bigl(1-t^{-1/2}\bigr)^{2H} \biggr]
\\
& \qquad \times \biggl[ \frac{1}{2} (a+b)^2
\bigl(1+t^{-2H}\bigr) -ab\bigl(1+t^{-1}\bigr)^{2H}-
\frac{a^2+b^2}{2}\bigl(1-t^{-1}\bigr)^{2H} \biggr].
\end{align*}

So
\begin{align*}
& \biggl[ \frac{1}{2} (a+b)^2\bigl(1+t^{-3H}\bigr)
-ab \bigl( 1+ 2H t^{-3/2} + H(2H-1)t^{-3} + o
\bigl(t^{-3}\bigr) \bigr)
\\
& \qquad -\frac{a^2+b^2}{2} \bigl( 1 - 2Ht^{-3/2} + H(2H-1)
t^{-3} + o\bigl(t^{-3}\bigr) \bigr) \biggr]
\\
& \qquad \times \bigl[ a^2 + b^2 - \bigl(2^{2H}-2
\bigr) ab \bigr]
\\
& \quad = \biggl[ \frac{1}{2} (a+b)^2\bigl(1+t^{-H}
\bigr) -ab \bigl( 1+ 2Ht^{-1/2} + H(2H-1) t^{-1} + o
\bigl(t^{-1}\bigr) \bigr)
\\
& \qquad -\frac{a^2+b^2}{2} \bigl( 1- 2Ht^{-1/2} + H(2H-1)
t^{-1} + o\bigl(t^{-1}\bigr) \bigr) \biggr]
\\
& \qquad \times \biggl[ \frac{1}{2} (a+b)^2 \bigl( 1+
t^{-2H}\bigr) -ab \bigl( 1 + 2H t^{-1} + H(2H-1)
t^{-2} + o\bigl(t^{-2}\bigr) \bigr)
\\
& \qquad -\frac{a^2+b^2}{2} \bigl( 1- 2Ht^{-1} + H(2H-1)t^{-2}
+ o\bigl(t^{-2}\bigr) \bigr) \biggr].
\end{align*}

{\it First case}: $0 < H < \frac{1}{2}$, $a+b \neq0$.
By Taylor's expansion we get, as $t \rightarrow\infty$,
\[
\frac{1}{2}(a+b)^2 \bigl[ a^2 + b^2 -
\bigl(2^{2H}-2\bigr) ab \bigr] t^{-3H} \approx\frac{1}{4}
(a+b)^4 t^{-3H},
\]
which is true if and only if
\[
\frac{(a-b)^2}{2} -\bigl(2^{2H}-2\bigr) ab = 0.
\]
However, it is easy to check that $\frac{(a-b)^2}{2} -(2^{2H}-2) ab>0 $
for fixed $b$ and every real~$a$.

\smallskip

{\it Second case}: $0 < H < \frac{1}{2}$ and $a+b = 0$.
By Taylor's expansion we get, as $t \rightarrow\infty$,
\[
\bigl[ a^2 + b^2 - \bigl(2^{2H}-2\bigr) ab
\bigr] t^{-3/2} \approx\bigl(-2Hab + \bigl(a^2+b^2
\bigr)H\bigr) t^{-3/2},
\]
which is true if and only if
$a=b = 0$, which is false.\vadjust{\eject}


{\it Third case}: $ \frac{1}{2} < H < 1$, $a-b \neq0$.
By \xch{Taylor's}{Taylor;s} expansion we get, as $t \rightarrow\infty$,
\[
H(a-b)^2 \bigl[ a^2 + b^2 -
\bigl(2^{2H}-2\bigr) ab \bigr] t^{-3/2} \approx
H^2(a-b)^4 t^{-3/2},
\]
which is true if and only if
\[
a^2(1-H) + b^2(1-H) + ab\bigl(2-2^{2H} + 2H
\bigr) = 0.
\]
However, it is easy to check that $ a^2(1-H) + b^2(1-H) + ab(2-2^{2H} +
2H)>0 $ for fixed $b$ and every real $a$.

{\it Fourth case}: $ \frac{1}{2} < H < 1$ and $a-b = 0$.
By Taylor's expansion we get, as $t \rightarrow\infty$,
\[
\frac{1}{2} (a+b)^2 \bigl[ a^2 + b^2 -
\bigl(2^{2H}-2\bigr) ab \bigr] t^{-3H} \approx
\frac{1}{4}(a+b)^4 t^{-3H},
\]
which is true if and only if
$2- 2^{2H} =0$, which is false since $H \neq\frac{1}{2}$. The proof of
Lemma~\ref{p:1} is complete.
\end{proof}

\section{Study of increments and some applications}

Let us start by the following lemma, in which we characterize the
second moment increments of the gfBm.

\begin{lemma}
\label{L:2}
For all $(s,t) \in{\Bbb R}_+^2$ such that $ s \le
t$\textup{:}
\begin{enumerate}
%
\item\ \vspace*{-23pt}
\begin{align*}
E \bigl( Z_t^H(a,b) - Z_s^H(a,b)\bigr) ^2 &= \bigl(a^2 + b^2\bigr) | t-s|^{2H}\\
&\quad - 2^{2H} ab \bigl( | t |^{2H} + | s |^{2H} \bigr) + 2ab | t + s |^{2H};
\end{align*}
\item\ \vspace*{-24pt}
\begin{align*}
\gamma(a,b,H) (t-s)^{2H} \le E \bigl( Z_t^H(a,b)- Z_s^H(a,b) \bigr) ^2 \le\nu(a,b,H)(t-s)^{2H},
\end{align*}
\end{enumerate}
where
\[
\gamma(a,b,H) = \bigl( a^2 + b^2 - 2 ab
\bigl(2^{2H-1}-1\bigr) \bigr) {\bf 1}_{{\cal C}}(a,b,H) + \bigl(
a^2 + b^2 \bigr) {\bf1}_{{\cal D}}(a,b,H),
\]
and
\[
\nu(a,b,H) = \bigl( a^2 + b^2 \bigr) {
\bf1}_{{\cal C}}(a,b,H) + \bigl( a^2 + b^2 - 2 ab
\bigl(2^{2H-1}-1\bigr) \bigr) {\bf1}_{{\cal D}}(a,b,H) ,
\]
\[
{\cal C} = \biggl\{ (a,b, H) \in{\mathbb R}^2\setminus\bigl\{ (0,0)
\bigr\} \times]0,1[; \biggl(H > \frac{1}{2}, ab \ge0 \biggr) \; or \;
\biggl( H < \frac{1}{2}, ab \le0\biggr) \biggr\} ,
\]
\[
{\cal D} = \biggl\{ (a,b, H) \in{\mathbb R}^2\setminus\bigl\{ (0,0)
\bigr\} \times]0,1[; \biggl(H > \frac{1}{2}, ab \le0\biggr) \; or \;
\biggl(H < \frac{1}{2}, ab \ge0 \biggr) \biggr\} .
\]

Moreover, the constants in the inequalities of statement 2 are the best possible.
\end{lemma}

\begin{proof}
The first statement is a direct consequence of Eqs.~(\ref{eq:1}) and
(\ref{eq:3}). So we will just check the second one. We will do that in
the case where $H > \frac{1}{2}$ and $ab
\ge0$; the proof in the remaining cases is similar.

Since the function $x \longmapsto x^{2H}$ is convex on
${\Bbb R}_+$, we have
\[
2^{2H-1} \bigl(t^{2H} + s^{2H} \bigr) -
(t+s)^{2H} \ge0,
\]
which yields that
\[
E \bigl( Z_t^H(a,b) - Z_s^H(a,b)
\bigr) ^2 \le \bigl(a^2 + b^2\bigr) | t-s
| ^{2H} .
\]

To get the lower bound, we consider the function
\[
f_{\lambda}: x \longmapsto2^{2H-1} \bigl((s+x)^{2H} +
s^{2H} \bigr) - (2s+x)^{2H} - \lambda x^{2H},
\]
with $\lambda> 0$. We have $ f_\lambda(0)
= 0$. So to get the stated lower bound, it suffices to check that, for
$\lambda= 2^{2H-1}-1$, $f_\lambda$ is a decreasing function. But here
we show more;
we look for the lower bound $\nu$ (resp.\ the upper bound $\gamma$)
of the set of reals $\lambda> 0$ such that $f_\lambda$ decreases
(resp.\ $f_\lambda$ increases) for $x > 0$, which will give us the
maximum $\gamma\ge0 $ and the minimum $\nu> 0$ such that
\[
\gamma(t-s) ^{2H} \le2^{2H-1} \bigl(t^{2H} +
s^{2H} \bigr) - (t+s)^{2H} \le\nu (t-s)^{2H}.
\]

The function $f_{\lambda}$ is differentiable in ${\Bbb R}_+^\star$,
and for all $x > 0$,
\[
f_{\lambda}'(x) = 2H \bigl[ 2^{2H-1}
(s+x)^{2H-1} -(2s +x)^{2H-1} - \lambda x^{2H-1} \bigr].
\]
So
\begin{align*}
f_\lambda'(x) < 0 \quad& \Longleftrightarrow \quad 2^{2H-1}
(s+x)^{2H-1} -(2s +x)^{2H-1} < \lambda x^{2H-1}
\\
& \Longleftrightarrow \quad\frac{2^{2H-1} (s+x)^{2H-1} -(2s +x)^{2H-1}}{x^{2H-1}} < \lambda.
\end{align*}

Denote $g(x) = \frac{2^{2H-1} (s+x)^{2H-1} -(2s +x)^{2H-1}}{x^{2H-1}}$, that is,
\[
g(x) = 2^{2H-1} \biggl( \frac{s}{x} + 1 \biggr) ^{2H-1} -
\biggl( 2 \frac{s}{x} + 1 \biggr)^{2H-1} .
\]

Putting $X = \frac{s}{x}$, we can write $ g(x) = h(X)$ where
\[
h(X) = 2^{2H-1} ( X + 1 ) ^{2H-1} - ( 2 X + 1 )^{2H-1}.
\]
The function $h$ is continuous and strictly decreasing. Consequently,
\[
h\bigl(]0, + \infty[ \bigr) = \,\Big] \lim_{X \rightarrow+ \infty} h ; \lim
_{X \rightarrow0} h \Big[\,= \,\big] 0; 2^{2H-1} -1 \big[ .
\]
So, $\sup_{x > 0} g = 2^{2H-1} -1$ and $\inf_{x > 0} g = 0$. We deduce that the lower bound $\nu> 0$ and the
upper bound $\gamma> 0$ such that
\[
\gamma(t-s) ^{2H} \le2^{2H-1} \bigl(t^{2H} +
s^{2H} \bigr) - (t+s)^{2H} \le\nu (t-s)^{2H}
\]
are $\nu= 2^{2H-1} -1$ and $\gamma= 0$.
\end{proof}

The first main application of Lemma~\ref{L:2} is the following result.
\begin{lemma}
For every $H \in \,] 0, 1  [\, \setminus \{ \frac
{1}{2}  \}$, the gfBm is not a semimartingale.
\end{lemma}

\begin{proof}
By Lemma $2.1$ of \cite{TB1} this result is a direct consequence
of Lemma~\ref{L:2}.
\end{proof}

The following result is the second important consequence of Lemma~\ref
{L:2}. It deals with the continuity and \xch{nondifferentiablity}{nondifferentabiity} of the gfBm
sample paths.

\begin{lemma}
\label{L:12} Let $H \in(0, 1)$.
\begin{enumerate}
\item The gfBm $Z^H$ admits a version whose sample paths are almost
surely H\"older continuous of order strictly less than $H$.
\item For every $H \in] 0; 1[^N$,
\begin{equation}
\label{eq:27} \lim_{\epsilon\rightarrow0^+} \sup_{t \in[t_0-\epsilon, t_0+
\epsilon]} \bigg|
\frac{ Z^H(t) - Z^H(t_0) }{t-t_0} \bigg| = + \infty,
\end{equation}
with probability one for every $t_0 \in{\mathbb R} $.
\end{enumerate}
\end{lemma}

\begin{proof}
The first statement follows by the Kolmogorov criterion from
Lemma~\ref{L:2}.

The second one is easily obtained by using the specific expression of
the second moment of the increments of the gfBm (\ref{eq:5}) and by
following exactly the same strategies as in the proofs of Lemma $4.1$
and Lemma $4.2$ in \cite{MZROSE}.
\end{proof}

As the third consequence of Lemma~\ref{L:2}, we see that if $b \neq
0$, then the gfBm does not have stationary increments, but this
property is replaced by the inequalities appeared in statement $2$ of
Lemma~\ref{L:2}. In order to understand how far is the gfBm from a
process with stationary increments, we will compare the gfBm increments
and the fBm increments. To reach this goal, let us first recall that,
in the fractional Brownian motion case $(a= 1, b=0)$, we have, for all
$p \in{\mathbb N}$ and $n \ge0$,
\begin{align*}
{\mathbb E} \bigl( \bigl(B_{p+1}^H -B_p^H
\bigr) \bigl(B_{p+n+1}^H-B_{p+n}^H\bigr)
\bigr) &= {\mathbb E} \xch{\bigl(B_{1}^H\bigl(B_{n+1}^H-B_{n}^H\bigr) \bigr)}{(\bigl(B_{1}^H\bigl(B_{n+1}^H-B_{n}^H\bigr) \bigr)}
\\
&=\frac{1}{2} \bigl[ (n+1)^{2H}-2n^{2H} +
(n-1)^{2H} \bigr]
\\
&=R_B(0,n) .
\end{align*}
Denote
\begin{equation}
R_Z(p,p+n)= {\mathbb E} \bigl( \bigl(Z_{p+1}^H
-Z_p^H\bigr) \bigl(Z_{p+n+1}^H
-Z_{p+n}^H\bigr) \bigr), \quad p \ge1 .
\end{equation}
In the following proposition, we compute the term $R_Z(p,p+n)$, showing
how different it is from $R_B(0,n)$.

\begin{prop}
For every $n \ge1$, we have, as $p \rightarrow\infty$,
\[
R_Z(p,p+n) = \bigl(a^2+b^2\bigr) R_B(0,n) -ab
\bigl( 2^{2H-1} H(2H-1) \bigr) p^{2(H-1)} \bigl(1 + o(1) \bigr),
\]
and, consequently,
\begin{equation}
\lim_{p\rightarrow\infty} R_Z(p,p+n) = \bigl(a^2+b^2
\bigr) R_B(0,n) .
\end{equation}
\end{prop}

\begin{proof}
An easy calculus allows us to get
\begin{align}
R_Z(p,p+n) & = \frac{a^2+b^2}{2} \bigl[(n+1)^{2H}-2n^{2H}
+ (n-1)^{2H} \bigr] \notag
\\
&\quad - ab \bigl[ (2p+n+2)^{2H} - 2(2p+n+1)^{2H}
+(2p+n)^{2H} \bigr] \notag
\\
& = \bigl(a^2+b^2\bigr) R_B(0,n) -ab
f_p(n),
\end{align}
for every $n \ge1$,
where $f_p(n) = (2p+n+2)^{2H} - 2(2p+n+1)^{2H} +
(2p+n)^{2H} $.

By Taylor's expansion we have, as $p \rightarrow\infty$,
\begin{align*}
f_p(n) &=(2p)^{2H} \biggl[ \biggl( 1+ \frac{n+2}{2p}
\biggr) ^{2H} - 2 \biggl( 1 + \frac{n+1}{2p} \biggr) ^{2H} +
\biggl( 1 + \frac{n}{2p} \biggr) ^{2H} \biggr]
\\
&= \bigl( 2^{2H-1} H(2H-1) \bigr) p^{2(H-1)} \bigl( 1+ o(1) \bigr)
.
\end{align*}
Since $H < 1$, the last term tends to $0$ as $p$ goes to infinity.
\end{proof}

\begin{remark}
If $a \neq0$ and $b \neq0$, then the gfBm increments are not
stationary. The meaning of the proposition is that they converge to a
stationary sequence.
\end{remark}

Now, we are interested in the behavior of the gfBm increments with
respect to $n$ (as $n \rightarrow\infty$) and, in particular, in the
long-range dependence of the process~$Z^H$.

\begin{definition}
We say that the increments of a stochastic process $X$ are long-range
dependent if for every integer $p \ge1$, we have
\[
\sum_{n \ge1} R_X(p, p+n) = \infty,
\]
where $R_X(p,p+n) = {\mathbb E} ( (X_{p+1}
-X_p)(X_{p+n+1} -X_{p+n})  )$.
\end{definition}


\begin{theorem}
For every $(a,b) \in{\mathbb R}^2 \setminus\{ (0,0) \}
$, the increments of $Z^H(a,b)$ are long-range dependent if and only if
$H > \frac{1}{2}$ and $a \neq b$.
\end{theorem}

\begin{proof}
For every integer $p \ge1$, by Taylor's expansion, as $n
\rightarrow\infty$, we have
\begin{align*}
R_Z(p,p+n) &= \frac{a^2+b^2}{2} n^{2H} \biggl[ \biggl( 1+ \frac{1}{n} \biggr)^{2H} -2 + \biggl( 1 - \frac{1}{n}\biggr) ^{2H} \biggr]\\
&\quad -ab n^{2H} \biggl[ \biggl( 1+ \frac{2p+2}{n}\biggr)^{2H}\!\! - 2 \biggl( 1+ \frac{2p+1}{n} \biggr)^{2H}\!\! +\biggl( 1+ \frac{2p}{n} \biggr)^{2H} \biggr]\\
&= H(2H-1)n^{2H-2} (a-b)^2\\
&\quad -4H(2H-1) (H-1) ab (2p+1)n^{2H-3} \bigl( 1 + o(1) \bigr) .
\end{align*}

If $a \neq b$, we see that as $n \rightarrow\infty$, $R_Z(p,p+n) \approx H(2H-1)n^{2H-2} (a-b)^2 $. Then
\[
\sum_{n \ge1} R_Z(p,p+n) = \infty
\quad\Longleftrightarrow\quad 2H-2 > -1 \quad\Longleftrightarrow\quad H > \frac{1}{2}.
\]

If $a=b$, then, as $n \rightarrow\infty$, $R_Z(p,p+n)
\approx 4H(2H-1)(1-H) a^2 (2p+1) n^{2H-3} $. For every $H \in]0, 1[$,
we have $2H-3 < -1$ and, consequently,
\[
\sum_{n \ge1} R_Z(p,p+n) < \infty.\qedhere
\]
\end{proof}

\begin{remark} We know that the subfractional Brownian motion
increments $Z^H( \frac{1}{\sqrt{2}}, \frac{1}{\sqrt{2}}) $
are short-range dependent for every $H \in]0,1[$. From the theorem we
see another important motivation of the investigation of the general
process $Z^H(a,b)$ with $a$ and $b$ not necessary equal; it allows us
to exhibit models taking into account not only short-range dependence
of the increments, but also phenomena of long-range dependence if it exists.
\end{remark}

Lemma~\ref{L:2} has also an immediate application to local time of the gfBm.

\begin{corollary}
For $0 < H < 1$, on each \textup{(}time\textup{)}-interval $[0,T]
\subset[0, + \infty[$, the process $(Z_t^H)_{0 \le t
\le T}$ admits a local time $L^H([0,T],x)$, which satisfies
\[
\int_{\Bbb R} L^H\bigl([0,T],x\bigr)^2\,dx < \infty.
\]
\end{corollary}

\begin{proof}
First, denote, for $s,t \in{\Bbb R}_+$ and $s \neq t$, by
$\varphi_{Z_t^H-Z_s^H} $ the characteristic function of the random
variable $Z_t^H - Z_s^H$ and by $p^H(x;t,s)$ its probability density
function. They are expressed by
\begin{equation}
\label{eq:4} \varphi_{Z_t^H-Z_s^H} (u) = {\Bbb E} \bigl[ \exp \bigl( iu
\bigl(Z_t^H-Z_s^H\bigr) \bigr)
\bigr] = \exp \biggl( - \frac{u^2}{2} {\Bbb E} \bigl( Z_t^H
- Z_s^H \bigr) ^2 \biggr),
\end{equation}
and
\begin{equation}
\label{eq:5} p^H(x;s,t) = 1/ \sqrt{ 2 \pi{\Bbb E} \bigl(
Z_t^H - Z_s^H \bigr)
^2} \exp \bigl( - x^2/ 2 {\Bbb E} \bigl(
Z_t^H - Z_s^H \bigr)
^2 \bigr) .
\end{equation}
It is clear that
\[
\int_{- \infty}^{\infty} \big| \varphi_{Z_t^H-Z_s^H} (u) \big| \,du <
\infty,
\]
and by Lemma~\ref{L:2} and the fact that $0 < H < 1$, for every $T > 0$,
\[
\int_0^T \int_0^T
p^H(0;s,t) \,ds \,dt \le c \int_0^T
\int_0^T | t-s |^{-H} \,ds\, dt <
\infty,
\]
where $c$ is positive constant.

So, by Lemma $3.2$ of \cite{Berman} the local time of $(Z_t^H)_{0 \le t \le T}$ exists almost surely, and it is square
integrable as a function of $u$.
\end{proof}

\section*{Acknowledgment}
The author thanks the anonymous reviewers for their careful reading
this manuscript and their insightful comments and suggestions.


\end{document}